\documentclass[11pt]{article}

\usepackage{amssymb}
\usepackage{amsthm}
\usepackage{amsmath}
\usepackage{amscd}
\usepackage{graphicx}
\usepackage{hyperref}

\allowdisplaybreaks
    
\newcommand{\R}{\mathbb{R}}

\newcommand{\Sph}{\mathbb{S}}

\newcommand{\eps}{\varepsilon}

\newcommand{\dif}{\mathrm{d}} 
\newcommand{\Dif}{\mathrm{D}}

\newcommand{\coan}{C^{0,\alpha}_{\nu - 2}}

\newcommand{\ctan}{C^{2,\alpha}_\nu}

\newcommand{\apsol}{\tilde \Sigma_r(\sigma, \delta, s)}
\newcommand{\difop}{\Phi_{r, \sigma, \delta, s}}

\newtheorem*{mainthm}{Main Theorem}
\newtheorem{thm}{Theorem}
\newtheorem{lemma}[thm]{Lemma}

\theoremstyle{definition}
\newtheorem{defn}[thm]{Definition}

\newtheoremstyle{rmk}{5pt}{5pt}{}{}{\scshape}{:}{.5em}{}
\theoremstyle{rmk}

\newcommand{\mylabel}
	{\label}
\begin{document}

\title{A Gluing Construction for Prescribed Mean Curvature} 

\author{Adrian Butscher \thanks{butscher@math.stanford.edu, Department of Mathematics, Stanford University, Stanford, CA 94305}}

\maketitle

\begin{abstract}
	The gluing technique is used to construct hypersurfaces in Euclidean space having approximately constant prescribed mean curvature.  These surfaces are perturbations of unions of finitely many spheres of the same radius assembled end-to-end along a line segment.  The condition on the existence of these hypersurfaces is the vanishing of the sum of certain integral moments of the spheres with respect the prescribed mean curvature function.
\end{abstract}

\renewcommand{\baselinestretch}{1.25}
\normalsize

\section{Introduction}

In the recent paper \cite{memazzeo}, Butscher and Mazzeo have constructed examples of constant mean curvature (CMC) hypersurfaces in a Riemannian manifold $M$ with axial symmetry by gluing together small spheres positioned end-to-end along a geodesic $\gamma$. The examples they have constructed have very large mean curvature $2/r$ and lie within a distance $\mathcal O(r)$ of either a segment or a ray of $\gamma$; hence we say that these surfaces \emph{condense} to the appropriate subset of $\gamma$. Such surfaces cannot exist in Euclidean space, and their existence relies on the fact that the gradient of the ambient scalar curvature of $M$ acts as a `friction term' which permits the usual analytic gluing construction (akin to the classical gluing constructions pioneered by Kapouleas \cite{kapouleas7},\cite{kapouleas6}) to be carried out.  The purpose of this paper is to show the same techniques used in \cite{memazzeo} can be adapted in a straightforward manner to show that a similar construction is possible in a much simpler yet fairly general context: that of hypersurfaces having near-constant \emph{prescribed} mean curvature in Euclidean space.  The essence of the gluing construction carried out herein therefore lies in identifying and appropriately exploiting the analogous `friction term' appearing in this setting.

Let $F: \R^{n+1} \times T\R^{n+1} \rightarrow \R$ be a given, fixed smooth function.  For simplicity and to maintain the parallel with \cite{memazzeo}, we will assume that $F$ has cylindrical symmetry in the following sense.  Endow $\R^{n+1}$ with coordinates $(x^0, x^1, \ldots, x^n)$ and let $G \subseteq O(n+1)$ be the set of orthogonal transformations that fix the $x^0$-axis.   Each rotation $R \in G$ acts on $T\R^{n+1}$ via the differential $R_\ast :  T\R^{n+1} \rightarrow T\R^{n+1}$. We will  now demand that $F(R(p) , R_\ast V_p) = F(p, V_p)$ for all $(p, V_p)  \in \R^{n+1} \times T\R^{n+1}$.  The prescribed mean curvature problem that will be solved in this paper is to find, for every sufficiently small $r \in \R_+$, a $G$-invariant hypersurface $\Sigma_r$ which satisfies
\begin{equation}
	\label{eqn:pmc}
	H[\Sigma_r](p) = 2 + r^2 F( p, N_{\Sigma_r}(p) ) \qquad \forall \; p \in \Sigma_r
\end{equation}
where $H[\Sigma_r]$ is the mean curvature of $\Sigma_r$ and $ N_{\Sigma_r}$ is the unit normal vector field of $\Sigma_r$.  Furthermore, this hypersurface will be constructed by gluing together a finite number $K$ of spheres of radius one (and thus of mean curvature exactly equal to two) whose centres lie on the $x^0$-axis  using small catenoidal necks having the $x^0$-axis as their axes of symmetry. 

In order to properly state the Main Theorem, we must make the following definition, which is meant to capture the most important effect of the prescribed mean curvature function $F$ on the surface that we're attempting to construct in this paper.

\begin{defn}
	\label{def:moment}
	Let $S$ be a compact surface in $\R^{n+1}$.  The \emph{$F$-moment} of $S$ is the quantity
	$$\mu_F(S) := \int_S F(x, N_S(x)) J \dif \mathrm{Vol}_S$$
	where $N_S$ is the unit normal vector field of $S$ and $\dif \mathrm{Vol}_S$ is the induced volume form of $S$, while $J : S \rightarrow \R$ is defined by $J(x) := \langle \frac{\partial}{\partial x^0} , N_S(x) \rangle$ for $x \in S$.
\end{defn}

\noindent Now let $p_k^0(s) := (s + 2(k-1), 0, \ldots, 0)$ and consider the spheres $S_k(s) := \partial B_1(p_k^0(s))$.  These spheres are positioned along the $x^0$-axis in such a way that each $S_k(s)$ makes tangential contact with $S_{k\pm1}(s)$.  The following theorem will be proved in this paper.

\begin{mainthm}
	Suppose that there is $s_0 \in \R$ such that
	\begin{itemize}
		\item the $F$-moments of the spheres $S_k(s_0)$ satisfy $\sum_{k=1}^K \mu_F (S_k(s_0) ) = 0 $
				
		\item the function $s \mapsto \sum_{k=1}^K \mu_F(S_k(s) )$ has non-vanishing derivative at $s = s_0$, 
	\end{itemize}	
	then there for all sufficiently small $r>0$, there is a smooth, embedded hypersurface $\Sigma_r$ which is a small perturbation of $\bigcup_{k=1}^{K} S_k(s_0)$ that satisfies the prescribed mean curvature equation \eqref{eqn:pmc}.
\end{mainthm}

It is easy to find a situation in which the conditions of the Main Theorem hold.  For example: if $F(\cdot, \cdot)$ is such that $\mu_F( \partial B_1 (x^0, x^1, \ldots, x^n) )$ is negative whenever $x^0$ is sufficiently negative and positive whenever $x^0$ is sufficiently positive, then the mean value theorem asserts that a zero of the function $s \mapsto \sum_{k=1}^K \mu_F(S_k(s) )$ can be found.  And if also $F(x, \cdot)$ is monotone as a function of $x^0$, then this function will have non-zero derivative.

An application of the Main Theorem, and indeed an inspiration for it, is the earlier work by Kapouleas on slowly rotating assemblies of water droplets \cite{kapouleas8}.  In this case, the prescribed mean curvature  function $F : \R^{n+1} \times T \R^{n+1} \rightarrow \R$ takes the form $F(p, N_{\Sigma_r}(p) ) := C(\omega) ( p^0 )^2$ where $p := (p^0, p^1, \ldots p^n)$ and $C(\omega)$ depends on the angular velocity $\omega$.  The prescribed mean curvature equation now approximates the effect of centrifugal force on the surface $\Sigma_r$ when $\omega$ is small.  One of the assemblies of water droplets that Kapouleas constructs is exactly as described in the Main Theorem.  (He constructs many other, more complex, and less symmetrical assemblies as well.)

Another application of the Main Theorem is for understanding the possible shapes an electrically charged soap film can adopt in the presence of a weak, axially symmetric electric field.  In this case, the equation satisfied by the surface adopted by the soap film is exactly \eqref{eqn:pmc}, where the prescribed mean curvature function $F : \R^{n+1} \times T \R^{n+1} \rightarrow \R$ takes the form $F(p, N_{\Sigma_r}(p) ) := -  C \langle \nabla \phi(p) ,  N_{\Sigma_r}(p) \rangle$ while $\phi: \R^{n+1} \rightarrow \R$ is the electric potential and $C$ is a constant.  We can see why this is so by writing the total energy of the soap film as the sum of a surface area term and a term proportional to the surface integral of $\phi$, and then computing the Euler-Lagrange equation for the variation of this energy subject to the constraint that the volume enclosed by the surface remains constant.   If we now assume that $\phi$ is such that the existence conditions of the Main Theorem hold, then the Main Theorem asserts that $K$ spherical, electrically charged soap films connected by small catenoidal necks can be held in equilibrium at special points in space by the electric field.

\section{The Approximate Solution}

To construct an approximate solution for the Main Theorem, we use essentially exactly the same procedure as in \cite[\S 3.1]{memazzeo}.  This will be outlined here very briefly for the convenience of the reader. The presentation is given for the dimension $n=2$ for simplicity; everything that follows can be easily adapted to the $(n+1)$-dimensional setting.

Endow $\R^{n+1}$ with coordinates $(x^0, x^1, \ldots, x^n)$ and let $\gamma$ be the arc-length parametrization of the $x^0$-axis with $\gamma(0) = (0, 0, \ldots, 0)$.  We will construct an approximate solution for the Main Theorem out of $K$ spheres of radius one as follows.  Choose a localization parameter $s \in \R$ and small separation parameters $\sigma_1, \ldots, \sigma_{K-1} \in \R_+$.  Define $s_1 : = s$ and $s_k := s + 2(k-1) + \sum_{l=1}^{k-1} \sigma_l$ for $k = 2, \ldots K$ and set $p_k := \gamma(s_k)$ and $p_k^{\pm} := \gamma(s_k \pm 1)$.  Define the spheres $S_k := \partial B_1( p_k )$.  These spheres will now be joined together according to the following three steps.  

\paragraph*{Step 1.}  The first step is to replace each $S_k$ with the surface $\tilde S_k$ obtained by taking the normal graph of a specially chosen function $G_k$ over $S_k  \setminus [ B_{\rho_k} (p_k^+) \cup B_ {\rho_k} ( p_k^-) ]$ where $\rho_k \in (0,1)$ is a small radius as yet to be determined.  The functions we use for this purpose satisfy the equations
\begin{itemize}
	
	\item $\mathcal L (G_k) = \eps_k^+ \delta(p_k^+) +  \eps_k^- \delta(p_k^-) + A_k J_k$ if $k = 2, \ldots K-1$
	
	\item $\mathcal L (G_1) = \eps_1^+ \delta(p_1^+) +   A_1 J_1$ if $k=1$
	
	\item $\mathcal L (G_K) = \eps_K^- \delta(p_K^-) + A_K J_K$ if $k = K$
	
\end{itemize} 
where $\mathcal L := \Delta_{\Sph^2} + 2$ is the linearized mean curvature operator of the unit sphere, the small scale parameters $\eps_k^\pm$ are yet to be determined and $\delta(q)$ is the Dirac $\delta$-function centered at $q$, while $J_k := \langle \frac{\partial}{\partial x^0}, N_{S_k} \rangle$ is the sole $G$-invariant function in the kernel of $\mathcal L$ normalized to have unit $L^2$-norm, and $A_k$ is chosen to ensure $L^2$-orthogonality to $J_k$.  Of course $J_k = x^0 \big|_{S_k}$, the restriction of the $x^0$ coordinate function to $S_k$.

\paragraph*{Step 2.} Let $W$ be the catenoid, i.e.~the unique complete minimal surface of revolution whose axis of symmetry is $\gamma$ and whose waist lies in the $(x^1, x^2)$-plane.  The next step is to find the truncated and re-scaled catenoidal neck of the form $W_k := B_{\rho_k'}(p_k^\flat) \cap \big[  \eps_k W + p_k^\flat + (\delta_k, 0, 0)\big]$ that fits optimally in the space  between $\tilde S_k$ and $\tilde S_{k+1}$ for $k = 2, \ldots, K-1$.  Here $\eps_k >0$ is a small scale parameter and $p_k^\flat$ is a point between $p_k^+$ and $p_{k+1}^-$ that are determined by the optimal fitting procedure while $\delta_k$ is a small displacement parameter that takes $W_k$ away from its optimal location and $\rho_k'$ is a small radius as yet to be determined.  The optimal fit is obtained by matching the asymptotic expansions of the functions giving $\tilde S_k \cap B_{\rho_k'}(p_k^\flat)$ and $\tilde S_{k+1} \cap B_{\rho_k'}(p_k^\flat)$ and $W_k$ as graphs over the translate of the $(x^1, x^2)$-plane passing through $p_k^\flat$ exactly as in \cite[\S 3.1]{memazzeo}.  One particularly important outcome of the matching is that $\eps_k$ from the previous step, as well as $ \eps_k^\pm$ and $p_k^\flat$ are all uniquely determined by $\sigma_k$. In fact, an invertible relationship of the form $\sigma_k := \Lambda_k(\eps_k)$ holds, with $\Lambda_k(\eps_k) = \mathcal O(\eps_k| \log( \eps_k)| )$. Finally, we find that we must choose $\rho_k, \rho_k' = \mathcal O(\eps_k^{3/4})$ to ensure the optimal fit between the necks and the perturbed spheres.

\paragraph*{Step 3.} The final step is to use cut-off functions to smoothly glue the neck $W_k(\sigma_k)$ into the space between $\tilde S_k$ and $\tilde S_{k+1}$.  The interpolating region is the annulus $B_{\rho_k'}(p_k^\flat) \setminus B_{\rho_k'/2}(p_k^\flat)$.

\bigskip
\medskip

\noindent In this way we obtain a family of surfaces depending on the $\sigma$, $\delta$ and $s$  parameters.   
\begin{defn}
	Let $K$ be given.  The approximate solution with parameters $\sigma := \{ \sigma_1, \ldots, \sigma_{K-1} \}$ and $\delta := \{ \delta_1, \ldots, \delta_{K-1}\}$ and $s$ is the surface given by
\begin{align*}
	\apsol &:= \left[ \bigcup_{k=1}^K \tilde S_k \right] \cup \left[ \bigcup_{k=1}^{K-1} \tilde W_k  \right]
\end{align*}
\end{defn}

\section{Solving the Projected Problem}

We now proceed to solve the equation \eqref{eqn:pmc} up to a finite-dimensional error term by perturbing the approximate solution constructed in the previous section.  The required analysis is in most respects identical to or less involved than the analysis found in \cite[\S 4 - \S 6]{memazzeo} and will thus again only be abbreviated here for the sake of the reader.  The outcome will be a surface $\Sigma_r^\sharp (\sigma, \delta, s)$ satisfying $H[ \Sigma_r^\sharp (\sigma, \delta, s)] - 2 - r^2 F \big|_{ \Sigma_r^\sharp (\sigma, \delta, s)} \in \tilde{ \mathcal W}$, where $\tilde{ \mathcal W}$ is a finite-dimensional space of functions that will be defined precisely below.  It arises because the linearized mean curvature operator, which governs the solvability of \eqref{eqn:tosolve}, possesses a finite-dimensional \emph{approximate kernel} consisting of eigenfunctions corresponding to small eigenvalues.   These small eigenvalues make it impossible to implement a convergent algorithm for prescribing the components of the mean curvature of the approximate solution lying in $\tilde {\mathcal W}$.

\paragraph*{Function spaces.} We first define the weighted H\"older spaces in which the analysis will be carried out.  These are essentially the same weighted spaces as in \cite[\S4]{memazzeo}, namely the spaces $C^{k,\alpha}_\nu (\apsol)$ consisting of all $C^{k,\alpha}_{\mathit{loc}}$ functions on $\apsol $ where the rate of growth in the neck regions of $\apsol$ is controlled by the parameter $\nu$.  Choose some fixed, small $0< R \ll 1$ and define a weight function $\zeta_{r} : \apsol\rightarrow \R$ as
\begin{equation*}
	\zeta_{r}(p) := 
	\begin{cases}
		\| x \| &\quad p = (x^0, x) \in  \bar B_{R \! /2}(p_k^\flat)  \:\: \mbox{for some} \:\:  k \\
		\mathit{Interpolation} &\quad p  \in  \bar B_R(p_k^\flat) \setminus B_{R \! /2}(p_k^\flat)  \:\: \mbox{for some} \:\:  k  \\
		1 &\quad \mbox{elsewhere} 
	\end{cases}
\end{equation*}
where the interpolation is such that $\zeta_r$ is smooth and monotone in the region of interpolation, has appropriately bounded derivatives, and is $G$-invariant.  Now for any open set $\mathcal U \subseteq \apsol$, define
\begin{equation*}
	| f |_{C^{k, \alpha}_\nu (\mathcal U)} :=   \sum_{i=0}^{k}  | \zeta_r^{i-\nu} \nabla^i f |_{0, \, \mathcal U} + [ \zeta_r^{k + \alpha -\nu} \nabla^k f ]_{\alpha, \,\mathcal U} 
\end{equation*}  
where $| \cdot |_{0, \mathcal U}$ is the supremum norm on $\mathcal U$ and $[ \cdot ]_{\alpha, \mathcal U}$ is the $\alpha$-H\"older coefficient on $\mathcal U$.  This is the norm that will be used in the $C^{k,\alpha}_\nu(\apsol)$ spaces.

\paragraph{The equation to solve.}  Let $\mu : \ctan (\apsol) \rightarrow \mathit{Emb}( \apsol, \R^{n+1})$ be the exponential map of $\apsol$ in the direction of the unit normal vector field of $\apsol$.  Hence $ \mu_{f} \big(\apsol \big)$  is the normal deformation of $\apsol$ generated by $f \in \ctan (\apsol)$.  The equation
\begin{equation}
	\label{eqn:tosolve}
	H \big[ \mu_{f} \big(\apsol \big) \big]  = 2 + r^2\, F \circ \left( \mu_f \times N_{\mu_{f}(\apsol)} \right)
\end{equation}
selects $f \in \ctan (\apsol)$ so that $\mu_{f}( \apsol)$ satisfies equation \eqref{eqn:pmc}.   In addition, the function $f$ will be assumed $G$-invariant.  Define the operator 
\begin{gather*}
	\Phi_{r, s, \sigma, \delta} : \ctan(\apsol) \rightarrow \coan(\apsol) \\
	\Phi_{r, s, \sigma, \delta} (f) := H \big[ \mu_{f} \big(\apsol \big) \big]  - 2 - r^2\, \mathcal F (f) 
\end{gather*}
where $\mathcal F(f) :=  F \circ \left( \mu_f \times N_{\mu_{f}(\apsol)} \right)$.  The linearization of $\difop$ at zero is given by
$$\mathcal L := \Dif \difop (0) = \Delta + \|B \|^2 + r^2\big( \Dif_1 F ( \mu_0 , N_{\apsol})  \cdot f N_{\apsol} - \Dif_2 F ( \mu_0, N_{\apsol})  \cdot \nabla f \big) $$
where $\Dif_1 F$ and $\Dif_2 F$ are the derivatives of $F$ in its first and second slots and $B := B[\apsol]$ is the second fundamental form of $\apsol$.  On the $k^{\mathit{th}}$ spherical part of $\apsol$, the operator $\mathcal L$ is a small perturbation of $\mathcal L_k := \Delta_{S_k} + 2$ which is the linearized mean curvature operator of the sphere $S_k$.  Let $J_k$ once again be the $G$-invariant function in its kernel and define the space 
$$\tilde {\mathcal W} := \mathrm{span} \{  \chi_{\mathit{ext}, k} J_k \,  , \,  \chi_{\mathit{ext}, k} \mathcal L_k (\eta_k ) : k = 1,  \ldots, K-1 \} \cup \{ \chi_{\mathit{ext}, K} J_K  \} \, . $$
Here $\chi_{\mathit{ext}, k}$ is a smooth cut-off function supported on $\tilde S_k$ and $\eta_k$ is a smooth cut-off function supported on the transition region between the $k^{\mathit{th}}$ neck and $\tilde S_k$ with the property that the support of $\nabla \eta_k$ and $\nabla \chi_{\mathit{ext}, k} $ do not overlap.  

We now prove the following theorem. Let   $\eps := \max \{ \eps_1, \ldots, \eps_{K-1} \}$ and $\delta := \max \{ \delta_1, \ldots , \delta_{K-1} \}$ and we will assume that $ \eps = \mathcal O( r^2 )$ and $\delta = \mathcal O(r)$, which will be justified \emph{a posteriori}.

\begin{thm}
	\label{prop:soluptocoker}
	If $r$ is sufficiently small, then there exists  $f:= f_r(\sigma,\delta,s)  \in \ctan(\apsol)$ with $\nu \in (1,2)$ so that 
	\begin{equation}
		\label{eqn:uptofindim}
		\difop(f) \in \tilde{\mathcal W} \, . 
	\end{equation}
	The estimate $|f|_{\ctan} \leq C r^2$ holds for the function $f$, where the constant $C$ is independent of $r$.  Finally, the mapping $(\sigma, \delta,s ) \mapsto f_r (\sigma, \delta, s)$ is smooth in the sense of Banach spaces.
\end{thm}

\begin{proof}

As in \cite{memazzeo}, we will use a fixed-point argument to solve the equation $\Phi_{r, \sigma, \delta, s} (f) \in \tilde{\mathcal W}$ for a function $f \in C^{2, \alpha}_\nu (\apsol) $ with $\nu \in (1,2)$.  The fixed-point argument follows from three steps: an estimate of the size of $\Phi_{r, s, \sigma, \delta, s}(0)$; the construction of a bounded parametrix $\mathcal R$ satisfying $\mathcal L \circ \mathcal R = \mathit{id} + \mathcal E$ where $\mathcal E : \coan(\apsol) \rightarrow \tilde{\mathcal W}$; and an estimate of the non-linear part of the operator $\difop$.  Each of these steps is given in great detail in \cite{memazzeo}.  Thus all that is needed here is to point out how the analysis of \cite{memazzeo} applies to the present situation. 

\paragraph*{Step 1.}  We begin with the estimate of $\left|  \difop (0)  \right|_{\coan} $, which is the amount that the approximate solution $\apsol$ deviates from being an actual solution of equation \eqref{eqn:tosolve}.  
This is done by adapting \cite[Prop.~13]{memazzeo}. In fact, by combining Steps 1, 2 and 4 of the estimate of $H \big[\apsol \big] - 2$ in the $C^{0,\alpha}_{\nu - 2}$ for $\nu \in (1,2)$ found in \cite[Prop.~13]{memazzeo}, with a straightforward estimate for the $\coan$ norm of the $r^2 \mathcal F$ term, we find that
$$\left|  \difop (0)  \right|_{\coan} \leq C \max \big\{ r^2  , \, \eps^{3/2 - 3\nu/4} , \delta  \eps^{1 - 3\nu/4} \big\} \leq  C r^2 $$
for some constant $C$ independent of $r$.
	
\paragraph{Step 2.}  We now find a parametrix $\mathcal R  : \coan(\apsol) \rightarrow \ctan(\apsol)$ satisfying $\mathcal L \circ \mathcal R = \mathit{id} + \mathcal E$ where $\mathcal E : \coan(\apsol) \rightarrow \tilde{\mathcal W}$.  As in \cite[Prop.~15]{memazzeo}, this is done by first constructing an approximate parametrix by patching together parametrices for the linearized mean curvature operator of each sphere with parametrices for the linearized mean curvature operator of each neck; and then iterating to produce an exact parametrix plus an error term in $\tilde{\mathcal W}$ in the limit.  The difference here is that the terms coming from the non-Euclidean background metric in \cite[Prop.~15]{memazzeo} must be replaced by the  $r^2 \mathcal F$ term.  The same result holds because this term can easily be shown to satisfy the right estimates.  In fact, $\mathcal R$ and $\mathcal E$ satisfy the estimate $| \mathcal R(w) |_{\ctan} + |\mathcal E(w)|_{C^{2,\alpha}_0} \leq C | w |_{\coan}$ for all $w \in \coan(\apsol)$, where $C$ is a constant independent of $r$.

\paragraph*{Step 3.} We define the quadratic and higher remainder term of the operator $\difop$ as
\begin{gather*}
	\mathcal Q : \ctan (\apsol) \rightarrow \coan (\apsol) \\
	\mathcal Q(f) := \difop(f) - \difop(0) - \mathcal L(f)
\end{gather*}
The estimates for the $C^{0,\alpha}_{\nu}$ norm of $\mathcal Q $ can be found exactly as in  \cite[Prop.~18]{memazzeo} with the terms coming from the non-Euclidean background metric replaced by the  $r^2 \mathcal F$ term.  Then there exists $M>0$ so that if  $f_1, f_2 \in \ctan(\apsol)$ for $\nu \in (1, 2)$ and satisfying $|f_1|_{C^{2,\alpha}_\nu} + |f_2|_{C^{2,\alpha}_\nu} \leq M$, then
	$$| \mathcal Q( f_1) - \mathcal Q( f_2) |_{\coan} \leq C  |f_1 - f_2 |_{\ctan} \max \big\{ |f_1|_{\ctan}, |f_2|_{\ctan} \big\}$$
where $C$ is a constant independent of $r$.  Once again, this works because the $r^2 \mathcal F$ term can easily be shown to satisfy the right estimates.

\paragraph*{Step 4.} We can now solve the CMC equation up to a finite-dimensional error term by implementing a fixed-point argument based on the parametrix constructed in Step 2 as well as the estimates we have computed so far.   Let $E :=  \difop(0) $ and use the \emph{Ansatz} $f := \mathcal R(w-E)$ to convert the equation $\difop(f) \in \tilde{\mathcal W}$ into the fixed point problem $w - \mathcal N_{r, \sigma, \delta, s}(w)  \in \tilde{\mathcal W}$ where
\begin{gather*}
	\mathcal N_{r, \sigma, \delta, s} : \coan(\apsol) \rightarrow \coan (\apsol) \\
	\mathcal N_{r, \sigma, \delta, s} (w) := - \mathcal Q \circ \mathcal R(w-E) + r^2 \mathcal F \circ \mathcal R(w-E) 
\end{gather*}
The estimates that have been established up to now give us the estimate 
\begin{align*}
	|\mathcal N_r(w_1) - \mathcal N_r(w_2) |_{\coan} &\leq C r^2 |w_1 - w_2 |_{\coan} 
\end{align*}
for $w$ belonging to a ball of radius $\mathcal O(r^2)$ about zero in $\coan(\apsol)$, where $C$ is independent of $r$.  Hence $\mathcal N_r$ is a contraction mapping on this ball if $r$ is sufficiently small, and  a solution of \eqref{eqn:uptofindim} satisfying the desired estimate can be found.  The smooth dependence of this solution on the parameters $(\sigma, \delta, s)$ is a consequence of the fixed-point process.
\end{proof}

\section{Force Balancing Arguments and the Proof of the Main Theorem}

When $r$ is sufficiently small, we have now found a function $f_r(\sigma, \delta) \in C^{2, \alpha}_\ast(\apsol)$ for each $(\sigma, \delta, s)$ so that 
$$H \big[ \mu_{f_r(\sigma, \delta)} \big(\apsol \big) \big]  - 2 - r^2 \mathcal F( f_r (\sigma, \delta, s)) = \mathcal E_r (\sigma, \delta, s) $$
where $\mathcal E_r (\sigma, \delta, s )$ is an error term belonging to the finite-dimensional space 
$\tilde{\mathcal W}$ depending on the free parameters $(\sigma, \delta, s)$.  The corresponding surface that satisfies the prescribed mean curvature condition up to finite-dimensional error is $\Sigma_r^\sharp( \sigma, \delta, s) := \mu_{f_r(\sigma, \delta, s)} ( \apsol)$.  

To complete the proof of the main theorem, we must show that it is possible to find a value of $(\sigma, \delta, s)$  for which these error terms vanish identically.  As in \cite[\S 7.2]{memazzeo}, we consider the \emph{balancing map} $B_r : \R^{2K-1} \rightarrow \R^{2K-1}$ defined by
\begin{equation}
	\label{eqn:balmapdef}
	B_r(\sigma, \delta, s) := \Big( \pi_1 \big( \mathcal E_r (\sigma, \delta, s)  \big), \ldots, \pi_{2K-1} \big( \mathcal E_r (\sigma, \delta, s)  \big) \Big)
\end{equation}
where $\pi_{2k + 1} : \tilde{\mathcal W} \rightarrow \R$  and $\pi_{2k} :  \tilde{\mathcal W} \rightarrow \R$ are the $L^2$ projection operators given by
$$\pi_{2k} (e) := \int_{\Sigma^\sharp_r(\sigma, \delta, s)} e \cdot \chi_{\mathit{neck}, k} I_k 
\qquad \mbox{and} \qquad 
\pi_{2k+1} (e) := \int_{\Sigma^\sharp_r(\sigma, \delta, s)} e \cdot \chi_{\mathit{ext}, k}^{\, \prime} J_k $$
where  $\chi_{\mathit{ext}, k}^{\, \prime}$ is a cut-off function supported on the $k^{\mathit{th}}$ spherical region,  $\chi_{\mathit{neck}, k}$ is a cut-off function supported on the $k^{\mathit{th}}$ neck and transition region, and $I_k$ the Jacobi field of the neck coming from translation along the neck axis.  This is an odd, bounded function with respect to the centre of the neck.  Note that $B_r$ is a smooth map between finite-dimensional vector spaces by virtue of the fact that the dependence of the solution $f_r(\sigma, \delta,s )$ on $(\sigma, \delta,s)$ is smooth and the mean curvature operator is a smooth map of the Banach spaces upon which it is defined.  The following lemma proves that $\pi(e) = 0$ implies that $e = 0$, and its proof is a straightforward computation.  (In order to make this work out, we must choose the cut-off functions properly: we must have overlap between the supports of $\chi_{\mathit{neck}, k}$, $\chi_{\mathit{ext}, k}$ and $\chi_{\mathit{ext}, k+1}$ and no overlap between the supports of  $\chi_{\mathit{ext}, k}^{\, \prime}$ and $\eta_k$.)

\begin{lemma}
	\label{lemma:projop}
	Choose $e \in \tilde{\mathcal W}$ as $e = \sum_{k=1}^K a_k \chi_{\mathit{ext}, k} J_k + \sum_{k=1}^{K-1} b_k \chi_{\mathit{ext}, k} \mathcal L_k (\eta_k )$ for $a_k, b_k \in \R$.  Then 
	\begin{align*}
		\pi_{2k} (e) &= C_1 b_k + C_1' ( \eps_{k+1}^{3/2} a_{k+1} - \eps_k^{3/2} a_k  ) \\
	\pi_{2k+1} (e) &= C_2  a_k 
	\end{align*}
	where $C_1, C_1', C_2$ are independent of $r$ and $(\sigma, \delta, s)$.
\end{lemma}

We must now show that $B_r (\sigma, \delta, s)$ can be controlled by the initial geometry of $\apsol$, at least to lowest order in $r$.  The calculations are similar to those found in \cite[\S 7.2]{memazzeo} except with the contributions from the ambient background geometry replaced by a contribution from the prescribed mean curvature in the form of the $F$-moments of the spheres making up $\apsol$.

The highest-order part of $\mathcal E_r (\sigma, \delta, s)$ involves the $F$-moments of the spherical constituents $S_k$ of $\apsol$ as follows.  Set $\mu_k (\sigma, s) := \mu_F (S_k)$ --- this depends on $s$ and $\sigma_1, \ldots, \sigma_k$ because the location of the centre of $S_k$ is determined by these parameters. Let us continue to assume that $\eps_k = \mathcal O(r^2)$ and $\delta_k = \mathcal O(r)$ for each $k$.  This will be justified shortly.

\begin{lemma}
	\label{prop:integralexp}
	The quantity $\mathcal E_r (\sigma, \delta, s)$ satisfies the formul\ae\
	\begin{subequations}
	\begin{align}	
		\pi_{2k} \big( \mathcal E_r (\sigma, \delta, s)  \big)  &=  C_1 \delta_k \eps_k^{3/2} + \mathcal O(r^{2 + 2 \nu}) 	\label{eqn:balform1} \\	
		\pi_{2k+1} \big( \mathcal E_r (\sigma, \delta, s)  \big)  &=   
		\begin{cases}
			C_2  \eps_{1}  - r^2 \mu_{1} (\sigma, s) + \mathcal O \big(r^ {4} \big)&\qquad k=0 \\
			C_2 \big( \eps_{k+1} - \eps_{k} \big)  - r^2 \mu_{k+1} (\sigma, s) + \mathcal O \big(r^ {4} \big) &\qquad 0 < k < K-1 \\
			- C_2  \eps_{K}   - r^2 \mu_{K} (\sigma, s) + \mathcal O \big(r^{4} \big)&\qquad k = K-1
		\end{cases} \label{eqn:balform2}
	\end{align}
	\end{subequations}
where $C_1, C_2$ are constants independent of $r, \sigma, \delta, s$.
\end{lemma}

\begin{proof}
Set $\Sigma_r^\sharp := \Sigma_r^\sharp (\sigma, \delta, s)$ and $\Sigma := \apsol$ for convenience.  Consider first equation \eqref{eqn:balform2} with $0 < k < K-1$.  By the first variation formula and estimates of the size of the perturbation generating $\Sigma_r^\sharp$ from $\apsol$, and calculating as in \cite[Prop.~27]{memazzeo}, we have
\begin{align*}
	\pi_{2k+1} \big( \mathcal E_r (\sigma, \delta, s)  \big)  &=   \int_{\Sigma_r^\sharp} \big( H [\Sigma_r^\sharp] - 2 - r^2 \mathcal F (f_r(\sigma, \delta, s) ) \big) \chi_{\mathit{ext}, k} J_k \\
	&= \int_{\partial \Sigma^\sharp \cap \mathit{supp}( \chi_{\mathit{ext}, k} ) } \Big\langle \frac{\partial}{\partial x^0} , \nu_k \Big\rangle - r^2 \int_{S_k}  F(x, N_{S_k}(x) ) ) \big) J_k + \mathcal O(r^4) \\
	&=  C_2 \big( \eps_{k+1} - \eps_{k} \big) - r^2 \mu_k(s, \sigma) + \mathcal O(r^4)
\end{align*}
where $\nu_k$ is the unit normal vector field of $\partial \Sigma^\sharp \cap \mathit{supp}( \chi_{\mathit{ext}, k} )$ in $\Sigma^\sharp$.

Consider now equation \eqref{eqn:balform2}.  In the neck we have $H[\apsol] = 0$.  Using similar estimates,
\begin{align*}
	\pi_{2k} \big( \mathcal E_r (\sigma, \delta, s)  \big)  &=   \int_{\Sigma_r^\sharp} \big( H [\Sigma_r^\sharp] - 2 - r^2 \mathcal F (f_r(\sigma, \delta, s) ) \big) \chi_{\mathit{neck}, k} I_k \\
	&= - 2  \int_{\Sigma \cap \mathit{supp}( \chi_{\mathit{neck}, k} ) } \chi_{\mathit{neck}, k} I_k + \mathcal O(r^{2 + 2 \nu}) \\
	&=  C_1 \delta_k \eps_k^{3/2} + O(r^{2 + 2 \nu})
\end{align*}
where $\delta_k$ is the displacement parameter of the $k^{\mathit{th}}$ neck.  This is because $I_k$ is an odd function with respect to the neck having $\delta_k = 0$, whereas the integral is being taken over the neck with $\delta_k \neq 0$.  Hence the integral $ \int_{\Sigma \cap \mathit{supp}( \chi_{\mathit{neck}, k} ) } \chi_{\mathit{neck}, k} I_k$ picks up the displacement of the $k^{\mathit{th}}$ neck from its position at $\delta_k = 0$. This same phenomenon arises in \cite[Prop.~27]{memazzeo}.
\end{proof}

\subsection{Proof of the Main Theorem}

It remains to find a value of the parameters $(\sigma, \delta, s)$ so that $\mathcal E_r(\sigma, \delta, s) = 0$.  As shown in Lemma \ref{lemma:projop}, this is equivalent to find a solution of the equation $B_r(\sigma, \delta, s) = 0$.  In what follows, we will continue to assume that $\eps = \mathcal O(r^2)$ and $\delta = \mathcal O(r)$ and this will be justified shortly.  As a consequence of Lemma \ref{prop:integralexp}, the equations that we must solve are as follows:
\begin{equation*}
	\begin{aligned}
		C_1 \delta_1 &= E_1(\sigma, \delta, s) \\
		\vdots \\
		C_{K-1} \delta_{K-1} &= E_{K-1}(\sigma, \delta, s) 
	\end{aligned}
	\qquad \mbox{and} \qquad 
	\begin{aligned}
		C_2 \eps_1 &= r^2 \mu_1(\sigma, s) + E_1'(\sigma, \delta, s) \\
		C_2 (\eps_2 - \eps_1) &= r^2 \mu_2 (\sigma, s) + E_2'(\sigma, \delta, s) \\
		\vdots \\
		C_2 (\eps_{K-1} - \eps_{K-2} ) &= r^2 \mu_{K-1}(\sigma, s) + E_{K-1}'(\sigma, \delta, s) \\
		- C_2 \eps_{K-1} &= r^2 \mu_K(\sigma, s) + E_{K}'(\sigma, \delta, s)
	\end{aligned}
\end{equation*}
where $\eps_k$ depends on $\sigma_k$ in an invertible manner as indicated in Step 2 of the construction of the approximate solution, and $E_k$, $E_k'$ are error quantities satisfying the bounds $|E_k| = \mathcal O(r^{-1 + 2 \nu})$ and $|E_k'| = \mathcal O(r^4)$.  We can abbreviate these equations by introducing the matrix $M := \left( \begin{smallmatrix} I & 0 \\ 0 & J \end{smallmatrix} \right)$
where $I$ is the $(K-1) \times (K-1)$ identity matrix and $J$ is the $K \times (K-1)$ matrix
$$J := \left( \begin{smallmatrix} 1 \\ -1 & 1 \\[-1ex] & &  \!\!\ddots \\	& & & \!\!\! -1 & 1 \\ & & &  &\!\!\! -1 \end{smallmatrix} \right) \, .$$
The equations become
\begin{equation}
	\label{eqn:firsteq}
	M (\delta, \eps)^t = (E, r^2 \mu + E')^t
\end{equation}
where $\delta := ( \delta_1, \ldots, \delta_{K-1})$, $\eps := (\eps_1, \ldots, \eps_{K-1})$ and so on for $E, E'$ and $\mu$.  

We will solve these equations in two steps as follows.  Note first that the matrix $M$ is injective but not surjective, with vectors in the image of $M$ satisfying the relation $(0, e) \cdot M (v, w) = 0$ for all $(v, w) \in \R^{2K-1}$, where $e := (1, 1, \ldots, 1)$. Let $\rho : \R^{2K-1} \rightarrow \R^{2K-2}$ be the orthogonal projection onto the image of $M$.  The equation 
\begin{equation}
	\label{eqn:secondeq}
	\rho M (\eps, \delta) = \rho( E, r^2 \mu + E')
\end{equation}
can now be solved using the implicit function theorem.  Thus the matrix $\rho M : \R^{2K-2} \rightarrow \R^{2K-2}$ is invertible and the equation at $r=0$ is just $\rho M (\eps, \delta) = 0$ which has the solution $(\eps, \delta) = 0$.  Hence the solution persists for small $r$.

We now have a solution $\eps := \eps_r(s)$ and $\delta_r(s)$ of \eqref{eqn:firsteq} for all sufficiently small $r$ and depending implicitly on the one remaining free parameter $s$. Moreover, we see that $\eps = \mathcal O(r^2)$ and $\delta = \mathcal O(r^{-1 + 2 \nu}) = \mathcal O(r)$ since $\nu \in (1, 2)$.  It remains to solve \eqref{eqn:secondeq} and we proceed as follows.  Once $(\eps, \delta)$  satisfy \eqref{eqn:secondeq}, then \eqref{eqn:firsteq} becomes equivalent to $ 0 = (0, e) \cdot M (\eps, \delta) = r^2 e \cdot \mu + e \cdot E' $ or simply
\begin{equation}
	\label{eqn:thirdeq}
	\sum_{k=1}^{K} \mu_k (\sigma_r(s), s ) + E''(\sigma_r(s), \delta_r(s), s) = 0
\end{equation}
where the error quantity satisfies the estimate $|E''| = \mathcal O(r^2)$.  

Equation \eqref{eqn:thirdeq} may or may not have a solution, depending on the nature of the function $\sum_k \mu_k$, which in turn depends on the specific nature of the prescribed mean curvature function $F$.  However, if the following two conditions are met, then the implicit function theorem guarantees the existence of a solution.  First, it must be the case that the equation at $r=0$ has a solution, in other words if the $F$-moments of the spheres $S_1, \ldots, S_K$ satisfy
$$\sum_{k=1}^K \mu_F(\partial B_1(p_k^0(s)) ) = 0 $$
for some $s$, where $p_k^0(s) := (s + 2(k-1), 0, \ldots, 0)$.  Second, if $s_0$ is the solution of this equation, then it must also be the case that the mapping
$$s \mapsto \sum_{k=1}^K \mu_F(\partial B_1(p_k^0(s)) )$$
has non-vanishing derivative at $s = s_0$.  If these conditions are satisfied, then the implicit function theorem 
implies that for $r$ sufficiently small, there is a solution $s(r)$ of \eqref{eqn:thirdeq}.  This completes the proof of the Main Theorem. \hfill \qedsymbol

\renewcommand{\baselinestretch}{1}
\small

\bibliography{pmc}
\bibliographystyle{amsplain}

\end{document}